\documentclass[reqno,twoside,a4paper,11pt]{amsart}

\usepackage{amsthm, amssymb, amsmath, amsfonts,eufrak,mathrsfs,dsfont}
\usepackage{url}
\usepackage[colorlinks=true,pdftex]{hyperref}
\hypersetup{linkcolor=red,citecolor=blue,urlcolor= blue}
\usepackage{verbatim}
\usepackage{color}
\usepackage[utf8]{inputenc}%utile pour taper directement les caractères accentués
\usepackage[displaymath, mathlines]{lineno}
\IfFileExists{epsf.def}{\input epsf.def}{\usepackage{epsf}}
\usepackage{mathtools}

\IfFileExists{myowntimes.sty}{%
\usepackage{mathpazo}
\usepackage{mathrsfs}
\usepackage[notref,notcite]{showkeys}
}
{}
\DeclareFontFamily{OT1}{eusb}{} \DeclareFontShape{OT1}{eusb}{m}{n} {<5> <6> <7> <8> <9> <10> <11> <12> <14.4> eusb10}{}
\DeclareMathAlphabet{\eusb}{OT1}{eusb}{m}{n}

\DeclareFontFamily{OT1}{eusm}{} \DeclareFontShape{OT1}{eusm}{m}{n} {<5> <6> <7> <8> <9> <10> <11> <12> <14.4> eusm10}{}
\DeclareMathAlphabet{\eusm}{OT1}{eusm}{m}{n}

\DeclareFontFamily{OT1}{eufm}{} \DeclareFontShape{OT1}{eufm}{m}{n} {<5> <6> <7> <8> <9> <10> <11> <12> <14.4> eufm10}{}
\DeclareMathAlphabet{\mathfrak}{OT1}{eufm}{m}{n}

\DeclareFontFamily{OT1}{fraktura}{}
\DeclareFontShape{OT1}{fraktura}{m}{n} {<5> <6> <7> <8> <9> <10> <11> <12> <13> <14.4> [1.1] eufm10}{}
\DeclareMathAlphabet{\fraktura}{OT1}{fraktura}{m}{n}

\DeclareFontFamily{OT1}{cmfi}{} \DeclareFontShape{OT1}{cmfi}{m}{n} {<5> <6> <7> <8> <9> <10> <11> <12> <13> <14.4> [0.9] cmfi10}{}
\DeclareMathAlphabet{\cmfi}{OT1}{cmfi}{b}{n}

\DeclareFontFamily{OT1}{cmss}{} \DeclareFontShape{OT1}{cmss}{m}{n} {<5> <6> <7> <8> <9> <10> <11> <12> <13> <14.4> cmss10}{}
\DeclareMathAlphabet{\cmss}{OT1}{cmss}{m}{n}
\DeclareMathAlphabet{\mathpzc}{OT1}{pzc}{m}{it}

\newtheoremstyle{thm}{1.8ex}{1.8ex}{\itshape\rmfamily}{} {\bfseries\rmfamily}{}{2ex}{}

\newtheoremstyle{def}{1.8ex}{1.8ex}{\slshape\rmfamily}{} {\bfseries\rmfamily}{}{2ex}{}

\newtheoremstyle{rem}{1.8ex}{1.8ex}{\rmfamily}{} {\bfseries\rmfamily}{}{2ex}{}

\theoremstyle{thm}

\newcommand\cour[1]{{\fontfamily{pcr}\selectfont #1}}

\theoremstyle{thm}
\newtheorem{theorem}{Theorem}[section]
\newtheorem{lemma}[theorem]{Lemma}

\newtheorem*{Main Theorem}{Main Theorem.}

\newtheorem*{special theorem}{Lindeberg-Feller Theorem for Martingales}

\theoremstyle{def}

\theoremstyle{rem}

\newtheorem{remarks}[theorem]{Remarks}

\numberwithin{equation}{section}

\renewcommand{\section}{\secdef\sct\sect}
\newcommand{\sct}[2][default]{%
\refstepcounter{section}
\addcontentsline{toc}{section}{{\tocsection {}{\thesection}{\!\!\!\!#1\dotfill}}{}}
\vspace{0.7cm}
\centerline{\scshape\thesection.\ #1} \nopagebreak \vspace{0.2cm}}
\newcommand{\sect}[1]{%
\vspace{0.4cm} \centerline{\large\scshape\rmfamily #1}
\vspace{0.2cm}}

\renewcommand{\subsection}{\secdef\subsct\sbsect}
\newcommand{\subsct}[2][default]{\refstepcounter{subsection}
\addcontentsline{toc}{subsection}
{{\tocsection{\!\!}{\hspace{1.2em}\thesubsection}{\!\!\!\!#1\dotfill}}{}}
\nopagebreak\vspace{0.45\baselineskip} {\flushleft\bf
\thesubsection~\bf #1.~}
\\*[3mm]\noindent
\nopagebreak}
\newcommand{\sbsect}[1]{\vspace{0.1cm}\noindent
\textbf{#1.~}\vspace{0.1cm}}

\renewcommand{\subsubsection}{%
\secdef \subsubsect\sbsbsect}
\newcommand{\subsubsect}[2][default]{%
\refstepcounter{subsubsection} 
\addcontentsline{toc}{subsubsection}{{\tocsection{\!\!}
{\hspace{3.05em}\thesubsubsection}{\!\!\!\!#1\dotfill}}{}}
\nopagebreak
\vspace{0.15\baselineskip} \nopagebreak {\flushleft\rmfamily
\itshape\thesubsubsection
\ \rmfamily #1\/.}\ }
\newcommand{\sbsbsect}[1]{\vspace{0.1cm}\noindent
\rmfamily \itshape
\arabic{section}.\arabic{subsection}.\arabic{subsubsection} \
\sffamily #1\/.\ }

\renewcommand{\caption}[1]{%
\vglue0.5cm
\refstepcounter{figure}
\begin{minipage}{0.9\textwidth}\small {\sc Figure~\thefigure. }#1\end{minipage}}

\def\myffrac#1#2 in #3{\raise 2.6pt\hbox{$#3 #1$}\mkern-1.5mu\raise 0.8pt\hbox{$#3/$}\mkern-1.1mu\lower 1.5pt\hbox{$#3 #2$}}
\newcommand{\ffrac}[2]{\mathchoice%
{\myffrac{#1}{#2} in \scriptstyle}
{\myffrac{#1}{#2} in \scriptstyle}
{\myffrac{#1}{#2} in \scriptscriptstyle}
{\myffrac{#1}{#2} in \scriptscriptstyle}
}

\definecolor{lightgray}{gray}{0.5}

\newcommand{\twocite}[2]{\cite{#1}--\cite{#2}}

\newcommand{\PP}        {\mathbb P}
\newcommand{\p}        {\cmss P}

\newcommand{\Z}        {\mathbb Z}

\newcommand{\E}        {\mathbb E}

\newcommand{\scrF}      {\mathscr{F}}

\newcommand{\AAA}         {\mathcal{A}}

\newcommand{\scrC}      {\mathscr{C}} 

\newcommand{\scrH}      {\mathscr{H}}

\newcommand{\scrG}      {\mathscr{G}}

\newcommand{\1}{\mathds 1}

\title[Anomalous heat kernel]{On heat kernel decay \\
for the random conductance model 
}
\author[ 
Boukhadra 
]
{
Omar~Boukhadra
}

\begin{document} 
%\thanks{\hglue-4.5mm\fontsize{9.6}{9.6}\selectfont\copyright\,2015 
%by Omar Boukhadra and Nina Gantert. Provided for non-commercial research and education use. 
%Not for reproduction, distribution or commercial use.\vspace{2mm}\\
%$^1$D\'epartement de Math\'ematiques, BP 325, route Ain El Bey, 25017, Constantine, Alg\'erie.
%}
%\thanks{$^2$3
%85748 Garching
%Deutschland.\\
%}

\maketitle
\vspace{-5mm}
\centerline{\it To the innocence of Timy and Kus}
\vspace{0.3cm}
\centerline{{Department of Mathematics},} 
\centerline{{University of Constantine 1}}
\centerline{\small\cour{\url{boukhadra@umc.edu.dz}}}

\smallskip

\vspace{-2mm}
\begin{abstract} 
\textit{
We study discrete time random walks in an environment of i.i.d. non-negative bounded conductances in $\Z^d$. We are interested in the anomaly of the heat kernel decay. We improve recent results and techniques.
}
\newline\\
\noindent
{\it\textbf{keywords}~:
Markov chains, Random walk, Random environments,
Random conductances,  Percolation.
}
\end{abstract}

%%%%%%%%%%%%%%%%%%%%%%%%%%%%%%%%%%%%%%%%%%%%%%%%%%%%%%%%%%%%%%
\section{Introduction and statement of the results}
%%%%%%%%%%%%%%%%%%%%%%%%%%%%%%%%%%%%%%%%%%%%%%%%%%%%%%%%%%%%%%

The present work focuses on the {Random Conductance Model} (RCM) with i.i.d. conductances on the grid $\Z^d$. Explicitly, let $E_d$ be the set of non oriented nearest-neighbor bonds, i.e. $E_d = \{ e = \{x, y\}: x, y \in \Z^d, | x - y | =1 \}$. Let $(\omega_e)_{e \in E_d}$ be a family of i.i.d. (non-negative) conductances defined on a probability space $( \Omega, \PP)$. The expectation with respect to $\PP$ is denoted by $\E$. We write $\omega_e = \omega_{xy} = \omega (x, y)$ and $x \sim y$ if $\{x, y\} \in E_d$. For any realization of the environment $\omega \in \Omega$, set
\begin{equation}
\label{tp}
\pi ( x ) = \sum_{y:| x - y | =1} \omega_{xy}, \qquad \p_\omega (x, y) = \frac{\omega_{xy}}{\pi (x)}, \qquad \pi(A) = \sum\limits_{ x \in A}\pi(x).
\end{equation}
%We use the notation $\pi(A) = \sum\limits_{ x \in A}\pi(x)$.

With these settings, we associate $X = (X_n)$ the Markov chain with state-space $\Z^d$ and transition probability $\p_\omega ( x, y)$ and let $P^x_\omega$ denote 
the \textit{quenched} law of $X$, started from $x$ with
%\begin{equation}
%P^x_\omega ( X_{n+1} = x_{n+1}|X_0=x_0, X_1 = x_1, \ldots ,X_n = x_n )= \p_\omega (x_n, x_{n+1})
%\end{equation}
%and
\begin{equation}
P^x_\omega ( X_0 = x ) = 1
\end{equation}
Let $\p^n_\omega$ denote the $n$-th power of the transition kernel $P_\omega$, i.e.,
\begin{equation}\label{transprob}
\cmss P^n_\omega (x, y) = P^x_\omega (X_n = y).
\end{equation}
$X$ is then a random walk in the random environment $\omega$. 
%The heat kernel is given by
%\begin{equation}
%p^\omega_n ( x, y) = \frac{P^x_\omega ( X_n =y )}{\pi (y)}.
%\end{equation}

Assume that $\PP (\omega_e > 0 ) > p_c (d)$ where $p_c (d)$ is the critical value for Bernoulli bond percolation on $\Z^d$, and let $\scrC$ be the infinite cluster where bonds are open if they have positive conductances (see \cite{G}). When $\PP (\omega_e > 0 ) $ is sufficiently close to one, the set $\Z^d \setminus \scrC$ is a union of finite clusters.
%These finite clusters are often called \textit{holes}. 

Let $o \in \Z^d$ denote the origin and define the conditional measure
\begin{equation}
\PP_o ( \cdot ) := \PP ( \cdot \mid o \in \scrC ).% \qquad \PP_\xi ( \cdot ) := \PP ( \cdot \mid o \in \scrC_\xi ).
\end{equation}

\medskip
The RCM model (see \cite{T}) has witnessed an intense research activity in recent years.
The particularity of the RCM is what we call anomalous behavior of the probability of return in the sense that its decay at time $2n$ 
is slower than the standard one, i.e. slower than $n^{- d/2}$. 
This phenomenon that depends on the environment law $\PP$ was first observed and investigated for the case of \textit{annealed} (averaged on the environment) return probabilities in Fontes and Mathieu \cite{Fontes-Mathieu}.
The reason for such  behavior for a random walk among random conductances is the trapping effect of clusters of small conductances in the environment (these clusters are often referred to as holes). Then, for the \textit{quenched} ($\PP$-a.s.) case, we have \cite{BBHK} where it was proved anomalous decay for $d \ge 5$ for heavy-tailed conductances. Afterwards, in  \twocite{B1}{BKM}, it has been shown that the transition from a normal decay with rate $n^{-d/2}$ to a slower decay happens in the class of power tail laws of the environment.
As for $d = 4$, there is still anomality which has been established in~\cite{BiBo}.

In the present work, our main result is the following which basically reconsiders and improves the result and techniques in \cite{B1}.

\begin{theorem}
\label{th}
Let $d \ge 2$ and $\alpha \in (0,1/2)$. Suppose that the environment is governed by bounded i.i.d. random conductances such that $\PP (\omega_e \in [0, 1]) = 1$ and 
\begin{equation}
\label{C}
\tag{C}
( 4 d - 2 ) \, \lim_{u \to 0} \frac{\log \PP \big(\omega_b \in [u, 2u]\big)}{\log u} <  \alpha,
\end{equation}
Then, there exists $c>0$, such that $\PP_o$-a.s. for any $n$ large enough, we have
\begin{equation}
\label{alb}
\p^{2n}_\omega (o,o) \ge c\,\pi (o)\, n^{- (2 + \alpha (d-1))}
\end{equation}
\end{theorem}

\medskip

\begin{remarks}\label{rem}
$(1)$ The lower bound \eqref{alb} is to be compared with general upper bounds from \cite[Theorem~2.1]{BBHK}, and the general lower 
bound $\p^{2n}_\omega (o,o) \ge c\, n^{- d/2}$ $($see for example \cite[Remark 1.3]{B1}$)$. Then we observe that we always have a normal decay in $d =2,3$ and our lower bound \eqref{alb} is interesting for $d \ge 5$ and 
$$\alpha < \frac12\, \frac{d-4}{d-1}$$ 
%(Note that the necessary choice of $\alpha \in (0,1/2)$ has not been precised in SPL version !) 

\medskip
\noindent
$(2)$ In the special case with i.i.d. polynomial lower tail conductances, i.e. 
\begin{equation}
\label{LP}
\tag{LP}
\PP ( \omega_e < u ) = u^\gamma ( 1 + o (1)), \quad u \to 0,
\end{equation} 
one can easily see that the condition \eqref{C} becomes $\gamma < {\alpha}/{(4 d -2)}.$
This last condition on $\gamma$ and the estimate \eqref{alb} yield
\begin{equation}
\liminf_n \frac{\log \p^{2n}_\omega (o,o)}{\log n} \ge - 2 - (d-1) (4d - 2) \gamma.
\end{equation}
\end{remarks}

\newpage
%%%%%%%%%%%%%%%%%%%%%%%%%%%%%%%%%%%%%%%%%%%%%%%%%%%%%%%%%%%%%%%%%%%%%%
\section{Proof of Theorem~\ref{th}}\label{S1}
%%%%%%%%%%%%%%%%%%%%%%%%%%%%%%%%%%%%%%%%%%%%%%%%%%%%%%%%%%%%%%%%%%%%%

%This section contains the proof of Theorem~\ref{th}. 
Henceforth, we let $d \ge 2$ and assume that the conductances $\{ \omega_e, e \in E_d \}$ are i.i.d. random variables with $\omega_e \in [0, 1]$ a.s. and $\PP (\omega_e > 0 ) > p_c (d)$. Let $\scrC$ denote the infinite cluster along positive conductances. (See \cite{G}).

Henceforth, we call a box centered at the origin $B_r$ of radius $r$ on $\Z^d$, the set
$$
B_r = [- r, r]^d \cap \Z^d,
$$
and $\partial B_r$ the frontier of $B_r$, that is, the sites of $B_r$ with a neighbor outside $B_r$.

Now pick a $\xi > 0$ such that $\PP (\omega_e \ge \xi ) > p_c (d)$. There exists then almost surely a unique infinite cluster of bonds with conductances larger than $\xi$ that we denote by $\scrC_\xi$.  It also represents the set of all sites in $\Z^d$ that have a path to infinity along edges with conductances at least $\xi$.
Note that $\scrC_\xi \subset \scrC.$ When $\xi$ is small enough, the components of $\scrC \setminus \scrC_\xi$ that we call \textit{holes} are finite a.s. and in a box of radius $n$, they have logarithmic size (see \cite[Lemma~4.1]{BKM}). For $x\in \scrC$, call $\scrH_x$ be the set of vertices in the finite component of $\scrC \setminus\scrC_\xi$ containing $x$.
The condition \eqref{C} supposes that the conductances law is heavy tailed and hence $\bigcup_{\xi>0} \scrC_\xi  = \Z^d.$

We are now interested in the expected time spent by the random walk in a box $B_r$. First,  we have to consider the so called coarse-grained walk $X^\xi$ which records the successive visits of $X$ to $\scrC_\xi$. Set $T_0 := 0$ and define
\begin{equation}
T_{\ell+1} := \inf \{ n > T_0 + \ldots+ T_\ell : X_n \in \scrC_\xi \} - (T_0 + \ldots+ T_\ell), \quad \ell \ge 1.
\end{equation}
For $\xi$ small enough, all components of $\scrC \setminus \scrC_\xi$ are finite a.s., and then $T_\ell < \infty$ a.s. for all $\ell$. 
These visits occur at the locations
\begin{equation}
X^\xi_\ell := X_{T_0 +\cdots+ T_\ell}.
\end{equation}
The sequence $(X^\xi_n)$ is a Markov chain on $\scrC_\xi$ with transition kernel given by
\begin{equation}
\p_\xi ( x, y) := P^x_\omega (X_{T_1} = y).
\end{equation}
We easily check that the restriction of the measure $\pi$ to $\scrC_\xi$ is invariant and reversible for the Markov chain on $\scrC_\xi$ induced by $\p_\xi$.
Set 
$\hat\tau_r = \inf \{ \ell \ge 0: X^\xi_\ell \notin B_r \}.%\qquad H_r = \inf \{ k \ge 0: X_k \in \partial B_r \}.
$ 
\begin{lemma}
\label{Etau^}
For almost every $\omega \in \{ o \in \scrC\}$, for $\xi$ small enough, there exists $c = c(d, \xi) < \infty $ such that for any $r$ large enough,
\begin{equation}
E^ o_\omega\big( \hat\tau_r\big) \le c\, r^{2}
\end{equation}
\end{lemma}
\begin{proof}
Let $r >0$ and set $n = c_* r^2$ where $c_*>0$ is chosen later. Choose $\xi$ small enough such that $\{o \in \scrC_\xi \}$ and \cite[Lemma~4.1]{BKM} holds. From \cite[Lemma~3.2]{BBHK}, the return probabilities decay of $X^\xi$ is standard. Thus for $x \in \scrC_\xi \cap B_r$, we have for $n$ large enough
\begin{equation}
\label{eteq}
P^x_\omega ( X^\xi_n \in B_r ) \le \sum_{y \in \scrC_\xi\cap B_r} \cmss P^n_\xi (x,y) \le c_1 c_2 n^{-d/2} r^d \le \frac12,
\end{equation}
where we choose $c_* \ge (2 c_1 c_2)^{2/d}$. Hence, $P^x_\omega ( \hat\tau_r > n ) \le 1/2$,
and the Markov property then gives, for $k$ a positive integer,
\begin{equation}
\label{taurec}
P^o_\omega ( \hat\tau_r > (k+1) n ) \le P^o_\omega \Big( P^{X^\xi_{kn}}_{\omega} ( \hat\tau_r > n ); \hat\tau_r > k n \Big)
 \le \frac12 P^o_\omega ( \hat\tau_r > k n ).  
\end{equation}
Hence by induction, we have $P^o_\omega ( \hat\tau_r > k n ) \le 1/2^k$,
which yields 
\begin{eqnarray*}
E^ o_\omega ( \hat\tau_r ) 
= 
\sum_{k\ge 0}  E^ o_\omega \big( \hat\tau_r; k n< \hat\tau_r \le (k+1) n \big) 
%&\le& 
%n\, \sum_k (k+1) \, P^ o_\omega \big( \hat\tau_r > k n\big) \\
\le
n\, \sum_k \frac{k+1}{2^k}
\le c\, n
\end{eqnarray*}
\end{proof}

Recall $\scrH_y$. Then, for any $x \in \scrC_\xi$, set 
\[\scrF_x =\bigcup_{y:\, \omega_{xy} >0} \scrH_y, \quad \scrF'_x = \{x\} \cup \scrF_x  \] 
and let 
$$
\scrG_x= \scrF'_x \bigcup \bigcup_{y \in \scrC_\xi, y \sim x} \scrF'_y.
$$ 
Note that $\scrG_x$ is the set of sites where $X$ can be hidden before stepping again onto the strong component $\scrC_\xi$.
Then, for
$
H_r = \inf \{ k \ge 0: X_k \in \partial B_r \}
$ 
we have:

\begin{lemma}
\label{Etau}
For $\PP$-a.e. $\omega \in \{o \in \scrC\}$, for $\xi>0$ small enough, there exists $c =c (d, \xi) < \infty$ such that for $r$ large enough, we have
\begin{equation}
\label{Etau1}
E^ o_\omega \big( H_{r}\big) \le c\, r^2\, \sup_{x \in \scrC_\xi \cap B_r} | \scrG_x |.
\end{equation}
\end{lemma}

\begin{proof}
Choose $\xi$ small enough such that Lemma~\ref{Etau^} holds. Then, observe that  
\begin{equation}
\label{tT}
H_r \le \tau_r \le \sum_{\ell=0}^{\hat\tau_r}  T_\ell.
\end{equation}
By the Markov property, we obtain therefore
\begin{equation}
E^ o_\omega \Big(\sum_{\ell=0}^{\hat\tau_r}  T_\ell\Big) 
\le E^ o_\omega \Big( \sum_{\ell=0}^{\hat\tau_r}  E^{X^\xi_\ell} (T_1)\Big)
\le 
\sup_{x \in \scrC_\xi \cap B_r} E^{x} (T_1)\,  E^ o_\omega \big(\hat\tau_r\big)
\end{equation}
Then use Lemma~\ref{Etau^} and \cite[Lemma~3.8]{BBHK} to get the desired estimate.
\end{proof}

For the following, we introduce a special configuration of the environment that we call \textit{trap}. 
Explicitly, let $\mathcal A_n$ be the event in the environment defined as follows: \textit{an edge $e = \{y,z\}$ with conductance $\omega_e \ge c > 0$, which we call the strong edge, is such that for any $e' \neq e$ incident with either $y$ or $z$, $\omega_{e'} \in [1/n,2/n]$}. The configuration $\mathcal A_n$ is then made up of a strong edge and of $4d - 2$ weak edges with conductances of order $1/n$.
Let $\AAA_n (x)$ be the event on the space of environments that $x$ is a vertex neighboring a trap
edge, which trap is situated outside the hypercubic box $B_{\max x_i}$, where the $x_i, i= 1,\ldots,d$ are the associated coordinates of $x$.
We say that the walk will meet a trap $\AAA_n$ at a time $k$ if we have $\AAA_n (X_k)$.

As we proceeded in recent works (see \twocite{B1}{BiBo}), to get the anomalous lower bound \eqref{alb}, we let the walk fall into a trap. 
The following lemma proves that the walk before exiting a box $B_{n^\alpha}$ will meet a.s. a trap $\AAA_n$.

\begin{lemma}
\label{trap-l}
Suppose that the condition \eqref{C} holds and set $r = n^\alpha$.
Then, there exist constants $c, \beta$ such that for $\PP$-a.e $\omega \in \{o \in \scrC\}$, there exists $N = N(\omega) < \infty$ a.s. such that for all $n \ge N$,
\begin{equation}
P^o_\omega \Big( \bigcap_{k=0}^{r/3-1} \AAA_n (X_{H_{3 k}})^c \Big) 
\le e^{- c n^\beta}.
\end{equation}
\end{lemma}

\begin{proof}
Let $\alpha > 0$ and set $r = n^\alpha$. 
Set $\overline \PP = \PP \times P^ o_\omega$. By construction, the events $\AAA_n (X_{H_{3k}})$ are $\overline \PP$-independent, and we have for every $k$,
\begin{equation}
\overline \PP  \big(\AAA_n (X_{H_{3k}})\big) \ge \PP (\AAA_n) \ge \PP (\omega_e \ge \ffrac12)\, 
\PP \big(\omega_e \in [\ffrac1n, \ffrac2n]\big)^{4d-2}.
\end{equation}
Therefore, it comes that for $n$ large enough
\begin{equation}
\label{4d-2}
\overline \PP \left(\bigcap_{k=0}^{r/3-3} \AAA_n (X_{H_{3 k}})^c\right) 
\le 
 \exp \Big( - c n^\alpha \PP \big(\omega_e \in [\ffrac1n, \ffrac2n]\big)^{4d-2} \Big). 
\end{equation}
Then set $p (\omega) = P^o_\omega \left(\bigcap_{k=0}^{r/3-3} \AAA_n (X_{H_{3 k}})^c\right)$ 
and observe that for positive $c', \beta$ that we choose later, Chebychev inequality yields
\begin{equation*}
\PP \Big( \omega : p (\omega) > e^{ - c' n^\beta} \Big) 
\le e^{c' n^\beta} \E \big(p (\omega)\big)
\le 
\exp \big( c' n^\beta - c n^\alpha \PP (\omega_e \in [\ffrac1n, \ffrac2n])^{4d-2} \big).
\end{equation*}
Thus the claim follows by Borel-Cantelli lemma if \eqref{C} holds and we choose in the last expression $c' = c/2$ and
$$
\beta = \alpha - ( 4 d - 2 ) \, \lim_{u \to 0} \frac{\log \PP ( \omega_b\in [u, 2 u] )}{\log u}.
$$ 
\end{proof}

Let $r\ge 3$ and define $K$ as the first rank such that $\AAA_n(H_{3k})$ happens, i.e.
$$
K = 
\inf \{ k \in \{0, \ldots,r/3-1\}: \AAA_n (X_{H_{3k}}) \}; \quad \text{with}\, \inf \emptyset = \infty.
$$

\begin{lemma}
\label{XinB}
Let $\alpha \in (0,1/2)$ and $r = n^\alpha$. Then $\PP_o$-a.e. $\omega$, there exists $c = c(d,\xi) > 0$ such that for $n$ large enough, we have
\begin{equation}
P^ o_\omega (X_n \in B_{3 (K+1)} \setminus B_{3 K}) \ge c \,n^{-1}
\end{equation}
\end{lemma}
\begin{proof}
Let $\{y,z\}$ be the strong edge of the trap $\AAA_n$ and define the events 
$$D_n =  \cap^{n}_{j =1} \big\{ X_j \in \{y,z\}\big\}.$$
Note that $P^ o_\omega (D_n) \ge ( 1 - (4d -2)/n)^n \ge e^{- (4d-2)}/2$.

If the walk crosses the edge of a trap $\AAA_n$ with conductance larger than $n^{-1}$, it will cost a probability larger than $1/(2dn)$.  Then the Markov property and Chebychev inequality give the following
\begin{eqnarray*}
P^ o_\omega \big(X_n \in B_{3 (K+1)} \setminus B_{3 K}\big) 
&\ge& 
E^ o_\omega \Big(\1_{\{H_{3K}< n\}} \, P^{X_{H_{3K}}}_\omega (X_{H_{3K}+1} = y) \, P^{y}_\omega \big( D_n\big); K < \infty \Big)\\
&\ge&
\frac{e^{-(4d-2)}/2}{2d}\, n^{-1}\, P^ o_\omega \big(H_{3K} < n, K < \infty\big)\\
&=&
\frac{e^{-(4d-2)}/2}{2d}\, n^{-1}\, \Big(P^ o_\omega \big(K < \infty\big) - n^{-1}\, E^ o_\omega \big(H_{3K}; K < \infty\big)\Big)
\end{eqnarray*}
On the one hand, remark that by Lemma~\ref{trap-l}, $P^ o_\omega \big(K < \infty\big)$ tends to $1$. 
On the other hand, from \cite[Lemma~4.1]{BKM} and Borel-Cantelli, for $n$ large enough
\begin{equation}
\label{hlog}
\sup_{x\in \scrC_\xi \cap B_{n^\alpha}} | \scrG_x | \le ( \log n )^{2d}.
\end{equation}
Then Lemma~\ref{Etau} yields that for $\xi$ small enough and $n$ large enough
\begin{equation}
n^{-1}\, E^ o_\omega \big(H_{3K}; K< \infty\big) \le n^{-1}\, E^ o_\omega \big(H_{r}\big) \le c\, (\log n)^{2d}\, n^{2\alpha -1}
\end{equation}
which tends to $0$ for $\alpha < 1/2$. And the claim follows.
\end{proof}

We are now ready for the proof of Theorem.
\medskip

\begin{proof}[Proof of Theorem~\ref{th}]
Let $\alpha \in (0,1/2)$. Set $r = n^\alpha$ and $B^\circ_k = B_{3 (k+1)} \setminus B_{3 k}$.  
Now observe that by the Markov property and reversibility, we have
\begin{equation}
\label{PR-CS}
\p^{2n}_\omega (o,o) \ge \sum_{k=0}^{r/3-1} \sum_{x \in B^\circ_{k}}  \p^{n}_\omega (o,x) \p^{n}_\omega (x,o)  
\ge \sum_{k=0}^{r/3-1} \sum_{x \in B^\circ_{k}}  \p^{n}_\omega (o,x)^2 \frac{\pi (o)}{\pi (x)}
\end{equation}
Bounding $\pi (x) \le 2d$ and  using Cauchy–Schwarz, we get
\begin{equation}
\label{PR-C}
\p^{2n}_\omega (o,o) \ge \frac{\pi (o)}{2 d}\, \sum_{k=0}^{r/3-1} | B^\circ_{k} |^{-1} P^ o_\omega \big(X_n \in B^\circ_{k}\big)^2.
\end{equation}
Since for all $k \in \{0, \ldots,r/3-1\}$, $| B^\circ_k | \le c n^{\alpha (d-1)}$, and using $K$, we obtain that 
\begin{equation}
\text{the r.h.s. of \eqref{PR-C}} \ge c\, \pi (o)\, n^{-\alpha(d-1)}\, P^ o_\omega \big(X_n \in B^\circ_{K}\big)^2.
\end{equation}
Then Lemma~\ref{XinB} yields the result.
\end{proof}

\section*{Acknowledgement}
%I would like to thank the referee for his comments. 
My sincere thanks to Nina, Pierre and Youcef Bey.

\end{document}